\documentclass[11pt, reqno]{amsart}

\usepackage{amsthm,amssymb,amstext,amscd,amsfonts,amsbsy,amsrefs,amsxtra,latexsym,amsmath,xcolor,mathrsfs,fancybox,upgreek, soul,url}
\usepackage[english]{babel}
\usepackage[all,cmtip]{xy}
\usepackage[latin1]{inputenc}
\usepackage{cancel}
\usepackage[draft]{hyperref}

\usepackage{comment}
\usepackage{mdframed}
\allowdisplaybreaks
\usepackage{mathtools}
\usepackage{enumerate}
\usepackage{thmtools}
\usepackage{thm-restate}
\usepackage{chngcntr}
\usepackage{enumitem}
\usepackage{etoolbox}
\usepackage{todonotes}

\usepackage{tikz}
\usepackage{verbatim}
\usetikzlibrary{quotes,angles}

\DeclarePairedDelimiter\abs{\lvert}{\rvert}
\DeclarePairedDelimiter\norm{\lVert}{\rVert}

\makeatletter
\let\oldabs\abs
\def\abs{\@ifstar{\oldabs}{\oldabs*}}
\let\oldnorm\norm
\def\norm{\@ifstar{\oldnorm}{\oldnorm*}}
\makeatother

\makeatletter
\glb@settings 
\makeatother

\usepackage[top=1.25in, bottom=1.25in, left=0.8in, right=0.8in]{geometry}

\newtheorem{theorem}{Theorem}
\newtheorem{lemma}[theorem]{Lemma}
\newtheorem{corollary}[theorem]{Corollary}
\newtheorem{proposition}[theorem]{Proposition}

\newtheorem*{result}{Result}

\theoremstyle{definition}

\theoremstyle{remark}

\newtheorem*{remark}{Remark}

\numberwithin{theorem}{section}
\numberwithin{proposition}{section}
\numberwithin{lemma}{section}
\numberwithin{corollary}{section}
\numberwithin{equation}{section}
\numberwithin{conjecture}{section}

\setlist[enumerate,1]{before=}
\AfterEndEnvironment{enumerate}{}

\newcommand{\Log}{\operatorname{Log}}
\newcommand{\N}{\mathbb{N}}
\newcommand{\Z}{\mathbb{Z}}
\newcommand{\R}{\mathbb{R}}
\newcommand{\C}{\mathbb{C}}
\newcommand{\Q}{\mathbb{Q}}

\usepackage{bm}

\def\H{\mathbb{H}}
\renewcommand{\pmod}[1]{\  \,  \left( \mathrm{mod} \,  #1 \right)}

\begin{document}

\author{Giulia Cesana}
\address{University of Cologne, Department of Mathematics and Computer Science, Division of Mathematics, Weyertal 86-90, 50931 Cologne, Germany}
\email{gcesana@math.uni-koeln.de}

\author{William Craig}
\address{Department of Mathematics, University of Virginia, Charlottesville, VA 22904}
\email{wlc3vf@virginia.edu}

\author{Joshua Males}
\address{Department of Mathematics, Machray Hall, University of Manitoba, Winnipeg,
	Canada}
\email{joshua.males@umanitoba.ca}

\title[Asymptotic equidistribution]{Asymptotic equidistribution for partition statistics and topological invariants}

\begin{abstract}
We provide a general framework for proving asymptotic equidistribution, convexity, and log-concavity of coefficients of generating functions on arithmetic progressions. Our central tool is a variant of Wright's Circle Method proved by two of the authors with Bringmann and Ono, following work of Ngo and Rhoades. We offer a selection of different examples of such results, proving asymptotic equidistribution results for several partition statistics, modular sums of Betti numbers of two- and three-flag Hilbert schemes, and the number of cells of dimension $a \pmod{b}$ of a certain scheme central in work of G\"{o}ttsche.
\end{abstract}

\maketitle

\section{Introduction and statement of results}
A {\it partition} $\lambda$ of a non-negative integer $n$ is a list of non-increasing positive integrers, say \linebreak $\lambda=(\lambda_1, \lambda_2, \dots, \lambda_m)$, that satisfies $|\lambda|\coloneqq\lambda_1+\dots+\lambda_m= n$, and we let $p(n)$ denote the number of such partitions.
In 1918, Hardy and Ramanujan \cite{HarRam} proved 
$$ p(n) \sim \frac{1}{4\sqrt{3}n} \cdot e^{\pi\sqrt{\frac{2n}{3}}}$$
as $n\to \infty$, one of the most famous asymptotic formulae in partition theory. Their work marked the birth of the so-called Circle Method.

Half a century later Wright \cite{Wright1,Wright2} developed a modified version of the Circle Method which provides a general method for studying the Fourier coefficients of functions with known asymptotic behaviour near cusps. The essence of Wright's method is to use Cauchy's theorem to recover the coefficients as the integral over a circle of the generating function. One then splits the integral into two arcs, the major arc and minor arc, where the generating function has large growth and small relative growth, respectively. Even though this version of the Circle Method gives weaker bounds than the original techniques of Hardy and Ramanujan, it is more flexible when working with non-modular generating functions. It has been used extensively in the literature, see e.g.\@  \cite{BringMahl,KimKimSeo,Mao} for several examples closely related to the present paper.

Throughout mathematics, the equidistribution properties of certain objects are a central theme studied by many authors, including in areas of algebraic and arithmetic geometry \cite{ComMar,GreenTao,Katz} and number theory \cite{OppShu,Xi}. Recently, there has been a body of work in analogy with Dirichlet's theorem on the asymptotic equidistribution (or non-equidistribution) on arithmetic progressions of various objects. For example, the third author showed the asymptotic equidistribution of the partition ranks in \cite{Males}, Ciolan proved asymptotic equidistribution results for the number of partitions of n into $k$-th powers in \cite{Ciolan}, Gillman, Gonzalez, Ono, Rolen, and Schoenbauer proved asymptotic equidistribution for Hodge numbers and Betti numbers of certain Hilbert schemes of surfaces \cite{GGOLS}, and Zhou proved asymptotic equidistribution of a wide class of partition objects in \cite{Zhou}.

Another example is one of the second author and Pun \cite{CraigPun}. We let $\mathcal{H}_t(\lambda)$ denote the multiset of \textit{$t$-hooks}, those hook lengths which are multiples of a fixed positive integer $t$, of a partition $\lambda$. They investigated the $t$-hook partition functions
\begin{align*}
	p_t^e(n)\coloneqq \# \{ \lambda \text{ a partition of } n    :   \# \mathcal{H}_t(\lambda) \ {\text {\rm is even}} \}, \quad
	p_t^o(n)\coloneqq \# \{ \lambda \text{ a partition of } n   :    \# \mathcal{H}_t(\lambda) \ {\text {\rm is odd}}\},
\end{align*}
which divide the partitions of $n$ into two subsets, those with an even (resp. odd) number of $t$-hooks.
For even $t$, they proved that the partitions of $n$ are asymptotically equidistributed between these two subsets, while for odd $t$ they found the surprising phenomenon that they are not. Following this example, Bringmann, Ono, and two of the authors \cite{BCMO} showed that on arithmetic progressions modulo primes $t$-hooks are not asymptotically equdistributed, while the Betti numbers of two specific Hilbert schemes are. Their results centrally used a variant of Wright's Circle Method (see Proposition \ref{WrightCircleMethod}). 

The primary aim of this paper is for proving large families of Dirichlet-type equidistribution theorems. We begin by making more precise the meaning of a Dirichlet-type theorem. Suppose $c(n)$ is an arithmetic function which counts something of interest. Let $q=e^{-z}$, where $z=x+iy\in\C$ with $x>0$ and $|y|<\pi$. Furthermore let $\zeta =\zeta_b^a\coloneqq e^{\frac{2\pi i a}{b}}$ be a $b$-th root of unity for some natural number\footnote{The case $b=1$ is clearly trivial for coefficients that are integral.} $b\geq 2$ and $0\leq a < b$. Assume that we have a generating function on arithmetic progressions $a \pmod{b}$ given by
\begin{align} \label{equation: generating function H(a,b;q)}
	H(a,b;q) = \sum_{n \geq 0} c(a,b;n) q^n,
\end{align}
for some coefficients $c(a,b;n)$ such that
\begin{align}\label{equation: splitting H(a,b;q)}
	H(a,b;q) = \frac{1}{b}\sum_{j=0}^{b-1} \zeta_b^{-aj} H\left(\zeta_b^j;q\right),
\end{align}
for some generating functions $H(\zeta;q)$, with $H(q)\coloneqq H(1;q)=\sum_{n\geq0} c(n)q^n$. To say that equidistribution holds is to say that $c(a,b;n) \sim \frac 1b c(n)$ as $n \to \infty$. We are concerned with relating analytic properties of the functions $H(\zeta;q)$ to equidistribution results for $c(a,b;n)$. We provide a general framework for answering this question for a large class of generating functions by applying the spirit of Wright's Circle Method along with ideas of \cite{BCMO} (see Theorem \ref{Thm: Equidistribution} for a precise statement). Since our aim is to unify differing approaches to asymptotic equidistribution, we also collect many known or partially-known results and prove them using our framework, which may be summarized as follows.

\begin{result}
Assume that on both the major and minor arcs $H(q)$ dominates $H(\zeta;q)$, and $H(q)$ is dominant on the major arc as $q\to 1$. Then $c(a,b;n)$ are eqidistributed as $n\to\infty$.
\end{result}

Theorem \ref{Thm: ranks} is already known, Theorem \ref{Thm: crank} is partially known, while (to the best of the authors' knowledge) Theorems \ref{Thm: residual crank}, \ref{Thm: plane partitions}, \ref{Thm: Betti}, and \ref{Thm: cells} are new.

Because this method also naturally produces asymptotic formulas for the coefficients $c(a,b;n)$, we may also derive other interesting results, namely results about convexity and log-concavity. Convexity-type results of partition theoretic objects have been studied in recent years, for example in \cite{BO} Bessenrodt and Ono showed that if $n_1,n_2 \geq 1$ and $n_1 +n_2 \geq 9$ then 
\begin{equation*}
	p(n_1) p(n_2) > p(n_1+n_2).
\end{equation*}
A similar phenomenon for partition ranks congruent to $a \pmod b$, denoted by $N(a,b;n)$, was investigated by Hou and Jagadeeson \cite{HJ}, who gave an explicit lower bound on $n$ for convexity of $N(a,2;n)$. Confirming a conjecture of \cite{HJ}, the third author showed in \cite{Males} that for large enough $n_1,n_2$ we have
\begin{align*}
	N(a,b;n_1)N(a,b;n_2) > N(a,b;n_1+n_2).
\end{align*}
A direct corollary to Proposition \ref{WrightCircleMethod} shows that $c(a,b;n)$ arising from functions that satisfy the conditions of Proposition \ref{WrightCircleMethod} also satisfy the convexity result
\begin{align}\label{eqn: convexity}
	c(a,b;n_1)c(a,b;n_2) > c(a,b;n_1+n_2)
\end{align}
for large enough $n_1,n_2$. A further corollary yields that the coefficients are \textit{asymptotically log-concave}, i.e., for large enough $n_1,n_2$,
\begin{align}\label{eqn: log concav}
	c(a,b;n)^2 \geq c(a,b;n-1)c(a,b;n+1).
\end{align}
Such log-concavity results have been obtained for various arithmetic coefficients in the literature, including \cite{BJMR,DawMas,DeSalvoPak} among many others. In particular, all of the coefficients discussed in the following sections asymptotically satisfy \eqref{eqn: convexity} and \eqref{eqn: log concav}. To the best of the authors' knowledge, this gives new results for the first residual crank, traces of plane partitions, Betti numbers of the two- and three-flag Hilbert schemes we consider, as well as the cells of the scheme $V_{n,k}$ of G\"{o}ttsche, each defined in the following subsections.

\subsection{Partition statistics}
We next consider various statistics on partitions, beginning with the asymptotic equidistribution properties of two of the most famous partition statistics: the rank and the crank. 

In \cite{Ramanujan} Ramanujan proved that for $n\geq 0$
\begin{align*}
p(5n+4) \equiv 0 \pmod{5},\qquad
p(7n+5) \equiv 0 \pmod{7},\qquad
p(11n+6) \equiv 0 \pmod{11}.
\end{align*}
The \textit{rank} \cite{dyson} of a partition $\lambda$ is given by the largest part minus the number of parts. Dyson \cite{dyson} conjectured, and Atkin and Swinnerton-Dyer \cite{Arken} later proved, that the partitions of $5n + 4$ (resp.\@ $7n + 5$) form $5$
(resp.\@ $7$) groups of equal size when sorted by their ranks modulo $5$ (resp.\@ $7$), thereby combinatorially explaining two of Ramanujan's congruences.
Moreover, Dyson posited the existence of another statistic which should explain all Ramanujan congruences, which he called the \textit{crank}. The crank was later found by Andrews and Garvan \cite{AndGa,Garvan}, and is given by
\begin{align*}
	\begin{cases*}
		\text{largest part of } \lambda  & \text{ if $\lambda$ contains no ones}, \\
		\mu(\lambda) - \omega(\lambda) & \text{ if $\lambda$ contains ones},
	\end{cases*}
\end{align*}
where $\omega(\lambda)$ denotes the number of ones in $\lambda$ and $\mu(\lambda)$ denotes the number of parts greater than $\omega(\lambda)$.

The function $N(a,b;n)$, which is the number of partitions of $n$ with rank congruent to $a\pmod{b}$, was shown to be asymptotically equidistributed by the third author in \cite{Males}, making use of Ingham's Tauberian theorem and monotonicity properties\footnote{Since the proof in \cite{Males} used Ingham's Tauberian theorem, there was no error term.}, and may also be concluded from \cite{Bringmann}. We reprove this result.

\begin{theorem}\label{Thm: ranks}
	Let $0\leq a <b$ and $b \geq 2$. Then as $n\to\infty$ we have that
	\begin{align*}
		N(a,b;n) =\frac{1}{b}p(n) \left(1+O\left(n^{-\frac{1}{2}}\right)\right).
	\end{align*}
\end{theorem}

In a similar vein, it is natural to consider the asymptotic behaviour of the crank on arithmetic progressions. For odd $b$, the asymptotic equidistribution is known by  Hamakiotes, Kriegman, and Tsai \cite{HaKrTs}, who used results on the asymptotic of cranks given by Zapata Rol\'{o}n in \cite{Rolon}. With our framework we are able to extend this result to all $b$. Note that our method is simpler than the full Circle Method, allowing us to easily extend to include the case of $b$ even. However, the asymptotic formulae obtained in \cite{HaKrTs} are far more precise than ours. Let $M(a,b;n)$ be the number of partitions of $n$ with crank congruent to $a \pmod{b}$.

\begin{theorem}\label{Thm: crank}
	Let $0\leq a <b$ and $b \geq 2$. Then as $n\to\infty$ we have that
	\begin{align*}
		M(a,b;n) = \frac{1}{b} p(n) \left(1+O\left(n^{-\frac{1}{2}}\right)\right).
	\end{align*}
\end{theorem}

In \cite{BriLo}, Bringmann, Lovejoy, and Osburn introduced two so-called residual cranks on overpartitions. Recall that an \textit{overpartition} is a partition where the first occurrence of each distinct number may be overlined. The \textit{first residual crank} of an overpartition is given by the crank of the subpartition consisting of the non-overlined parts. Let $\overline{M}(a,b;n)$ denote the number of overpartitions of $n$ whose first residual crank is congruent to $a\pmod{b}$.
\begin{theorem}\label{Thm: residual crank}
	Let $0\leq a <b$ and $b \geq 2$. Then as $n\to\infty$ we have that
	\begin{align*}
		\overline{M}(a,b;n) = \frac{1}{8bn}e^{\pi \sqrt{n}} \left(1+O\left(n^{-\frac{1}{2}}\right)\right).
	\end{align*}
\end{theorem}
\begin{remark}
	One could obtain a similar result for the second residual crank of \cite{BriLo}, which we omit here for succinctness.
\end{remark}

Our framework applies to a larger realm than just the classical theory of partitions. In fact, we now demonstrate an example where we can prove equidistribution in congruence classes for a plane partition statistic. A \textit{plane partition} of $n$ (see e.g.\@ \cite{Andrews}) is a two-dimensional array $\pi_{j,k}$ of non-negative integers $j,k\geq1$, that is non-increasing in both variables, i.e., $\pi_{j,k} \geq \pi_{j+1,k}$, $\pi_{j,k} \geq \pi_{j,k+1}$ for all $j$ and $k$, and fulfils  $|\Lambda| \coloneqq  \sum_{j,k} \pi_{j,k} = n$. For example there are six plane partitions of $3$, which we list below using the standard visual representation of plane partitions.
\begin{align*}
	\begin{matrix}
		1 & 1 & 1 \quad \quad  & 1 & 1 \quad \quad & 1			  & \quad \quad 2 & 1 \quad \quad & 2              &  \quad \quad 3 \\
		&   &                 & 1 &                & 1            &   			  &                & 1             &   \\
		&   &                 &   &                & 1            &   			  &                &                 &
	\end{matrix}
\end{align*}
We let $\operatorname{pp}(n)$ denote the number of plane partitions of $n$, so $\operatorname{pp}(3) = 6$. Plane partitions were famously studied by MacMahon \cite{MacMahon}, who established the generating function
\begin{align*}
	\operatorname{PP}(q) \coloneqq  \sum_{n=0}^\infty \operatorname{pp}(n) q^n = \prod_{n=1}^\infty \dfrac{1}{(1 - q^n)^n} = 1 + q + 3q^2 + 6q^3 + 13q^4 + 24q^5 + \cdots.
\end{align*}
As with regular partitions, many authors have studied asymptotic properties of families of plane partitions and their statistics. For instance, in 1931 Wright \cite{Wright3} established the asymptotic formula
\begin{align} \label{WrightPlanePartitions}
	\operatorname{pp}(n) \sim \dfrac{\zeta(3)^{\frac{7}{56}}}{\sqrt{12\pi}} \left( \dfrac{n}{2} \right)^{-\frac{25}{36}} \exp\left( 3 \zeta(3)^{\frac{1}{3}} \left( \dfrac{n}{2} \right)^{\frac{2}{3}} + \zeta^\prime(-1) \right)
\end{align}
as $n \to \infty$, where $\zeta(s)\coloneqq \sum_{k=1}^\infty \frac{1}{k^s}$ with $\operatorname{Re}(s)>1$ is the \textit{Riemann zeta function}. One of the more famous statistics associated the plane partition $\Lambda = \{ \pi_{j,k} \}_{j,k \geq 1}$ is its {\it trace} $t(\Lambda)$, which is defined by
$$t(\Lambda) = \sum_{j=1}^\infty \pi_{j,j}.$$
In \cite{Stanley}, Stanley generalized MacMahon's generating function to a two-variable function which keeps track of the values of $t(\Lambda)$, proving
\begin{align*}
	\sum_\Lambda \zeta^{t(\Lambda)} q^{|\Lambda|} = \prod_{n=1}^\infty \dfrac{1}{(1 - \zeta q^n)^n}.
\end{align*}
Certain asymptotic properties of the trace have been studied by Kamenov and Mutafchiev \cite{Kamenov-Mutafchiev} and Mutafchiev \cite{Mutafchiev}, where the limiting distribution and expected value of $t(\Lambda)$ are considered. Here, we study the distribution of the trace in residue classes. In particular, for integers $0 \leq a < b$ we define the function $\operatorname{pp}(a,b;n)$ as the number of plane partitions of $n$ whose trace is congruent to $a \pmod{b}$, that is,
\begin{align*}
	\operatorname{pp}(a,b;n) \coloneqq  \# \{ \Lambda : |\Lambda| = n, t(\Lambda) \equiv a \pmod{b} \}.
\end{align*}
For example, from the plane partitions of 3 given above we can see that $\operatorname{pp}(0,2;3) = 2$ and $\operatorname{pp}(1,2;3) = 4$.
\begin{theorem}\label{Thm: plane partitions} 
	Let $0\leq a < b$ and $b \geq 2$. Then as $n \to \infty$ we have that
	\begin{align*}
		\operatorname{pp}(a,b;n) \sim \frac{1}{b} \operatorname{pp}(n) \sim \frac{1}{b} \dfrac{\zeta(3)^{\frac{7}{56}}}{\sqrt{12\pi}} \left( \dfrac{n}{2} 					\right)^{-\frac{25}{36}} \exp\left( 3 \zeta(3)^{\frac{1}{3}} \left( \dfrac{n}{2} \right)^{\frac{2}{3}} + \zeta^\prime(-1) \right).
	\end{align*}
\end{theorem}

There are a plethora of other partition statistics in the literature for which one could obtain similar theorems using our framework. For example, such results could be proved for more residual crank-like statistics \cite{JS}, ranks for overpartition pairs \cite{BriLo2}, or the full rank of $k$-marked Durfee symbols \cite{BGM}.

\subsection{Betti numbers of Hilbert schemes}
 In topology a fundamental goal is to determine whether two spaces have the same topological, differential, or complex analytic structure. Topological invariants are important tools for determining when spaces have different structure. A prominent example are Betti numbers, which count the dimension of certain vector spaces of differential forms of a manifold. Often, the generating function of the Betti numbers are related to modular forms. Two prominent examples were investigated by Bringmann, Ono, and two of the authors in \cite{BCMO}, where it was shown that the Betti numbers of the Hilbert scheme of $n$ points on $\C^2$ as well as its quasihomogenous counterpart are each (essentially) asymptotically equidistributed\footnote{Here we mean equidistributed up to a trivial modification which comes from the fact that certain Betti numbers in this setting are identically zero. See the definition of $d(a,b)$ as below.} as $n \to \infty$. Here we provide further examples of this phenomenon.

For a Hilbert scheme $X$, let $b_j(X) \coloneqq \dim(H_j(X,\Q))$ be the \textit{Betti numbers}. Here, $H_j(X,\Q)$ denotes the $j$-th homology group of $X$ with rational coefficients. Then the generating function in a formal variable $T$ for the Betti numbers is known as the \textit{Poincar\'{e} polynomial}, defined by\footnote{The reader should be aware that often the Poincar\'{e} polynomial is written in the formal variable $T^{\frac{1}{2}}$, which explains some apparent mismatches between the referenced sources for generating functions in Section \ref{Section: proofs} and those quoted in this paper.}
\begin{align*}
	P(X;T) \coloneqq \sum_{j} b_j(X) T^j = \sum_{j} \dim(H_j(X,\Q)) T^j.
\end{align*}
We consider the modular sums of Betti numbers on congruence classes $a \pmod{b}$, and define
\begin{align*}
	B(a,b;X) \coloneqq \sum_{j \equiv a \pmod{b}} b_j(X).
\end{align*}
Define the three-step flag Hilbert scheme by
\begin{align*}
	X_1 \coloneqq	\text{Hilb}^{n,n+1,n+2}(0) = \left\{\C[[x,y]] \supset I_n \supset I_{n+1} \supset I_{n+2}  \colon  I_k \text{ ideals with } \dim_\C^{\C[[x,y]]} / I_k =k\right\},
\end{align*}
and the two-step flag scheme
\begin{align*}
	X_2 \coloneqq \text{Hilb}^{n,n+2}(0) = \left\{\C[[x,y]] \supset I_n \supset I_{n+2}  \colon  I_k \text{ ideals with } \dim_\C^{\C[[x,y]]} / I_k =k\right\}.
\end{align*}
 Furthermore, let $(J,I)$ be a point in 
 \begin{align*}
 	\text{Hilb}^{n,n+2}\left(\C^2\right) \coloneqq \left\{ I_{n} \in \text{Hilb}^{n}\left(\C^2\right), I_{n+2} \in \text{Hilb}^{n+2}(\C^2) \colon I_{n} \supset I_{n+2} \right\},
 \end{align*}
where $\text{Hilb}^n(\C^2)$ denotes the usual Hilbert scheme of $n$ points over $\C^2$. Then $J,I$ are said to be \textit{trivially related} if $J/I \cong \C^2$ as trivial $\C[x,y]$ modules (see \cite[Definition 4.2.1]{Boc}). We also consider
 \begin{align*}
 	X_3 \coloneqq \text{Hilb}^{n,n+2}\left(\C^2\right)_{\text{tr}},
 \end{align*}
 which is the subspace of $\text{Hilb}^{n,n+2}(\C^2)$ of trivially related points (see also \cite{NaYo}).
For $m\in \N$, we also regard the certain perverse coherent sheaves (defined explicitly in \cite{NaYo2}), called $X_4 \coloneqq \widehat{M}^m(c_N)$ where $c_N$ is some prescribed homological data.

Let 
\begin{align*}
		d(a,b)\coloneqq \begin{cases} \frac{1}{b} \ \ \ \ \ &{\text {\rm if $b$ is odd,}}\\
		\frac{2}{b} \ \ \ \ \ &{\text {\rm if $a$ and $b$ are even,}}\\
		0 \ \ \ \ \ &{\text {\rm if $a$ is odd and $b$ is even.}}
	\end{cases}
\end{align*}
We prove the following result, which shows that the Betti numbers of these schemes are (essentially) asymptotically equidistributed.

\begin{theorem}\label{Thm: Betti}
	Let $0\leq a < b$ and $b \geq 2$. Then as $n\to \infty$ we have that
	\begin{align*}
	\frac{1}{2}	B(a,b;X_1) \sim B(a,b;X_2) \sim B(a,b;X_3)  = \frac{d(a,b) \sqrt{3}}{4\pi^2} e^{\pi\sqrt{\frac{2n}{3}}} \left(1+O\left(n^{-\frac{1}{2}}\right)\right)
	\end{align*}
and 
\begin{align*}
	 B(a,b;X_4) \sim \frac{d(a,b) n^{\frac{m-2}{2}}}{6^{\frac{1-m}{2}} 2 \sqrt{2}c_m \pi^m} e^{\pi\sqrt{\frac{2n}{3}}} \left(1+O\left(n^{-\frac{1}{2}}\right)\right)
\end{align*}
where $\prod_{j=1}^{m} \frac{1}{1- e^{-jz}} = \frac{1}{c_m z^m} + O(z^{-m+1})$.
\end{theorem}
\begin{remark}
	It is possible to obtain further terms in the asymptotic expansion directly from the application of Theorem \ref{Thm: Equidistribution}, which  highlights the difference in lower-order terms of $B(a,b;X_j)$. Moreover, for $a$ odd and $b$ even, one may easily show that $B(a,b;X_j)$ identically vanish.
\end{remark}

Since many generating functions for topological invariants arise as infinite $q$-products, one may conclude similar results for many other functions. For example, in \cite{ManRol} Manschot and Zapata Rol\'{o}n investigated the asymptotics of the $\chi_y$-genera of Hilbert schemes of $n$ points on $K3$ surfaces, centrally using Wright's Circle Method. Since their generating function is a quotient of infinite $q$-products (see \cite[page 2]{ManRol}), it is likely that one may conclude similar equidistribution properties for these genera.

\subsection{A particular scheme of G\"{o}ttsche}

Let $\text{Hilb}_n(S)$ denote the Hilbert scheme which parametrises finite subschemes of length $n$ on a smooth projective surface $S$. We follow Fulton \cite{Fulton} and Ellingsrud--Str\o mme \cite{ES}, and say that a scheme $X$ has a \textit{cellular decomposition} if there is a filtration \linebreak $X =X_n \supset X_{n-1} \supset \dots \supset X_0 \supset X_{-1} = \varnothing $ by closed subschemes with each $X_{j} -X_{j-1}$ a disjoint union of schemes $U_{ij}$ isomorphic to certain affine spaces. Then the $U_{ij}$ are known as the \textit{cells} of the decomposition. 

Let $k$ be an algebraically closed field. Let $\bm{m}$ be the maximal ideal in $k[[x,y]]$, and define
\begin{align*}
	V_{n,k} \coloneqq \text{Hilb}_n\left( \text{spec}\left(  k[[x,y]] /\bm{m}^n\right) \right).
\end{align*}
The scheme $V_{n,k}$ was a central tool of G\"{o}ttsche in obtaining the famous formula for the Betti numbers of any Hilbert scheme of points on a smooth projective variety \cite{Got}, via the Weil conjectures. Let $v(a,b;n)$ count the number of cells of $V_{n,k}$ whose dimension is congruent to $a \pmod{b}$. 
\begin{theorem}\label{Thm: cells}
 Let $0 \leq a < b$ and $b \geq 2$. As $n\to \infty$ we have that
	\begin{align*}
		v(a,b;n) = \frac{1}{b}p(n) \left(1+O\left(n^{-\frac{1}{2}}\right)\right).
	\end{align*}
\end{theorem}

The paper is structured as follows. In Section 2 we recall relevant results from previous works in the literature. In Section 3 we then state our central theorem on the asymptotic equidistribution of coefficients of certain generating functions and show how convexity and log-concavity immediately follow from the asymptotics produced by Wright's Circle Method. Finally we prove the remaining theorems in Section 4.

\section*{Acknowledgements}
The second author thanks the support of the Thomas Jefferson Fund and the NSF (DMS-1601306 and DMS-2055118). The research of the third author conducted for this paper is supported by the Pacific Institute for the Mathematical Sciences (PIMS). The research and findings may not reflect those of the Institute. The authors would like to thank Walter Bridges for useful discussions on plane partitions as well as Kathrin Bringmann and Johann Franke for helpful comments on an earlier version of the paper.

\section{Preliminaries}

\subsection{Asymptotics of infinite $q$-products}

Here we recall the asymptotic behaviour of various infinite $q$-products.  One helpful tool is the modularity of the partition generating function
\begin{align*}
P(q) \coloneqq  \sum_{n=0}^\infty p(n)q^n = \frac{1}{(q;q)_\infty} = \frac{q^{\frac{1}{24}}}{\eta(\tau)},
\end{align*}
where we set $(a)_j =(a;q)_j \coloneqq  \prod_{\ell=0}^{j-1}(1-aq^\ell)$ for $j\in\N_0\cup \{\infty\}$, $q=e^{2\pi i\tau}$ and the Dedekind eta-function
$$\eta(\tau) \coloneqq  q^{\frac{1}{24}} \prod_{n=1}^{\infty} (1-q^n),$$
which is a modular form of weight $\frac 12$.
We also have \cite[Lemma 3.5]{BringDous}.
	\begin{lemma} \label{Lem: asymptotic P minor arc}
		Let $M>0$ be a fixed constant. Assume that $\tau=u+iv\in\H$, with $Mv\leq |u| \leq \frac 12$ for $u>0$ and $v\to 0$. We have that
		\begin{align*} 
			|P(q)| \ll \sqrt{v} \exp \left[\frac 1v \left(\frac{\pi}{12} -\frac{1}{2\pi}\left(1-\frac{1}{\sqrt{1+M^2}}\right)\right)\right].
		\end{align*}
	\end{lemma}
	\noindent This gives us the asymptotic behaviour of $P(q)$ on the so-called minor arc.

	Using the transformation property of $\eta$ we obtain the following classical asymptotic behaviour (see e.g.\@ \cite[equation (2.5)]{BCMO} with $k=1,h=0$ and shifting $z\mapsto \frac{z}{2\pi}$)
	\begin{align} \label{equ: asymptotic P major arc}
		\left(e^{-z}; e^{-z}\right)_\infty = \left(\frac{2\pi}{z}\right)^{\frac 12} e^{\frac{\pi}{12}\left(\frac{z}{2\pi}-\frac{2\pi}{z}\right)} \left(e^{-\frac{4\pi^2}{z}},e^{-\frac{4\pi^2}{z}} \right)_\infty,
	\end{align}
	for $z\in\C$ with $\operatorname{Re}(z)>0$.

Keeping the naming convention of \cite{BCMO}, we let
\begin{align*}
	F_1(\zeta; q)\coloneqq \prod_{n=1}^{\infty}\left(1-\zeta q^n\right), \qquad 
	F_3(\zeta; q)\coloneqq \prod_{n=1}^{\infty}\left(1-\zeta^{-1}(\zeta q)^n\right).
\end{align*}
Recall \textit{Lerch's transcendent}
\begin{align*}
	\Phi(z,s,a)\coloneqq \sum_{n=0}^\infty \frac{z^n}{(n+a)^s},
\end{align*}
and for $0\leq \theta < \frac{\pi}{2}$ define the domain $D_{\theta} \coloneqq \left\{ z=re^{i\alpha} \colon r \geq 0 \text{ and } |\alpha| \leq \theta \right\}$. Throughout, the \textit{Gamma function} is defined by $\Gamma(x)\coloneqq \int_0^\infty t^{x-1} e^{-t} dt $, for $\operatorname{Re}(x)>0$.
Then we have \cite[Theorem 2.1]{BCMO} (see also \cite{BFG} for the first case) which enables us to determine the asymptotics of $F_1$ and $F_3$ on major arcs.
\begin{theorem} \label{Thm: asymptotics F}
	For $b\geq 2$, let $\zeta$ be a primitive $b$-th root of unity. Then the following are true.
	\begin{enumerate}[leftmargin=*]
		\item[\rm (1)] As $z \to 0$ in $D_\theta$, we have 
		\begin{align*}
			F_{1}\left(\zeta;e^{-z}\right)  =\frac{1}{\sqrt{1-\zeta}} \, e^{-\frac{\zeta\Phi(\zeta,2,1)}{z}}\left( 1+O\left(|z|\right) \right).
		\end{align*}
		
		\item[\rm (2)] As $z\to 0$ in $D_\theta$, we have
		\begin{align*}
			F_3\left(\zeta;e^{-z}\right)= \frac{\sqrt{2\pi} \left(b^2z\right)^{\frac 12-\frac 1b}}{\Gamma\left(\frac{1}{b}\right)}
			\prod_{j=1}^{b-1}\frac{1}{(1-\zeta^j)^{\frac jb}}
			e^{-\frac{\pi^2}{6b^2z}}\left( 1+  O\left(|z|\right) \right).
		\end{align*} 
	\end{enumerate}
\end{theorem}

\subsection{Wright's Circle Method}

We require the following variant of Wright's Circle Method, which was proved by Bringmann, Ono, and two of the authors \cite[Proposition 4.4]{BCMO}, following work of Wright \cite{Wright2}, see also Ngo and Rhoades \cite{NgoRhoades}.

\begin{proposition} \label{WrightCircleMethod}
	Suppose that $F(q)$ is analytic for $q = e^{-z}$ where $z=x+iy \in \C$ satisfies $x > 0$ and $|y| < \pi$, and suppose that $F(q)$ has an expansion $F(q) = \sum_{n=0}^\infty c(n) q^n$ near 1. Let $N,M>0$ be fixed constants. Consider the following hypotheses:
	
	\begin{enumerate}[leftmargin=*]
		\item[\rm(1)] As $z\to 0$ in the bounded cone $|y|\le Mx$ (major arc), we have
		\begin{align*}
			F(e^{-z}) = z^{B} e^{\frac{A}{z}} \left( \sum_{j=0}^{N-1} \alpha_j z^j + O_M\left(|z|^N\right) \right),
		\end{align*}
		where $\alpha_j \in \C$, $A\in \R^+$, and $B \in \R$. 
		
		\item[\rm(2)] As $z\to0$ in the bounded cone $Mx\le|y| < \pi$ (minor arc), we have 
		\begin{align*}
			\lvert	F(e^{-z}) \rvert \ll_M e^{\frac{1}{\mathrm{Re}(z)}(A - \kappa)},
		\end{align*}
		for some $\kappa\in \R^+$. 
	\end{enumerate}
	If  {\rm(1)} and {\rm(2)} hold, then as $n \to \infty$ we have for any $N\in \R^+$ 
	\begin{align*}
		c(n) = n^{\frac{1}{4}(- 2B -3)}e^{2\sqrt{An}} \left( \sum\limits_{r=0}^{N-1} p_r n^{-\frac{r}{2}} + O\left(n^{-\frac N2}\right) \right),
	\end{align*}
	where $p_r \coloneqq  \sum\limits_{j=0}^r \alpha_j c_{j,r-j}$ and $c_{j,r} \coloneqq  \dfrac{(-\frac{1}{4\sqrt{A}})^r \sqrt{A}^{j + B + \frac 12}}{2\sqrt{\pi}} \dfrac{\Gamma(j + B + \frac 32 + r)}{r! \Gamma(j + B + \frac 32 - r)}$. 
\end{proposition}

\section{The central theorem}\label{section: central theorem}
Recall that we have the functions $H(a,b;q) $, $H(q)$, and $H(\zeta;q)$ as in \eqref{equation: generating function H(a,b;q)} and \eqref{equation: splitting H(a,b;q)}. We now prove a theorem regarding asymptotic equidistribution of the coefficients $c(a,b;n)$.
\begin{theorem} \label{Thm: Equidistribution}
Let $H(a,b;q)$ and $H(\zeta; q)$ be analytic on $|q|<1$, $|\zeta| = 1$ such that
$$H(a,b;q) = \dfrac{1}{b} \sum_{j=0}^{b-1} \zeta_b^{-aj} H(\zeta_b^j; q).$$
Suppose $c(a,b;n), c(n)$ are the Fourier coefficients of $H(a,b;q), H(1;q)$ respectively. Let $C = C_n$ be a sequence of circles centered at the origin inside the unit disk with radii $r_n \to 1$ as $n \to \infty$ that loops around zero exactly once. For $0<\theta $, let $\widetilde{C} \coloneqq C \cap D_\theta$ and $C \backslash \widetilde{C}$ be arcs such that the following hypotheses hold.
\begin{enumerate}[leftmargin=*]
\item[\rm(1)] As $z \to 0$ outside of $D_\theta$, we have
\begin{align*}
\sum_{j=1}^{b-1} \zeta_b^{-aj} H(\zeta_b^j; e^{-z}) = O\left( H(1; e^{-z}) \right).
\end{align*}
\item[\rm(2)] As $z \to 0$ in $D_\theta$, we have for each $1 \leq j \leq b-1$ that
\begin{align*}
H(\zeta_b^j; e^{-z}) = o\left( H(1; e^{-z}) \right).
\end{align*}
\item[\rm(3)] As $n \to \infty$, we have
\begin{align*}
c(n) \sim \dfrac{1}{2\pi i} \int_{\widetilde{C}} \dfrac{H(1;q)}{q^{n+1}} dq.
\end{align*}
\end{enumerate}
Then as $n \to \infty$, we have
$$c(a,b;n) \sim \dfrac{1}{b} c(n).$$
In particular, if $H(1;q)$ and $H(\zeta;q)$ satisfy the conditions of Proposition \ref{WrightCircleMethod} we have that
\begin{align*}
c(a,b;n) \sim \dfrac{1}{b} c(n) \sim \dfrac{1}{b} n^{\frac{1}{4}(- 2B -3)}e^{2\sqrt{An}} \left( \sum\limits_{r=0}^{N-1} p_r n^{-\frac{r}{2}} + O\left(n^{-\frac N2}\right) \right)
\end{align*}
as $n \to \infty$.
\end{theorem}

\begin{proof}
By Cauchy's theorem and the decomposition of $H(a,b;q)$ we have
\begin{align*}
c(a,b;n) = \dfrac{1}{2\pi i} \int_C \dfrac{H(a,b;q)}{q^{n+1}} dq = \dfrac{1}{b} \left[ \dfrac{1}{2\pi i} \int_C \dfrac{\sum_{j=0}^{b-1} \zeta_b^{-aj} H(\zeta_b^j; q)}{q^{n+1}} dq \right].
\end{align*}
We now break down the integral over $C$ into the components $\widetilde{C}$ and $C\backslash\widetilde{C}$. Along $C\backslash\widetilde{C}$, we have by (1) that
\begin{align*}
\dfrac{1}{2\pi i} \int_{C\backslash\widetilde{C}} \dfrac{\sum_{j=0}^{b-1} \zeta_b^{-aj} H(\zeta_b^j; q)}{q^{n+1}} dq = O\left( \dfrac{1}{2\pi i} \int_{C\backslash\widetilde{C}} \dfrac{H(1;q)}{q^{n+1}} dq \right).
\end{align*}
From (3) along with Cauchy's integral formula for $c(n)$ that
it follows that
\begin{align*}
\dfrac{1}{2\pi i} \int_{C\backslash\widetilde{C}} \dfrac{H(1;q)}{q^{n+1}} dq = o\left( \dfrac{1}{2\pi i} \int_{\widetilde{C}} \dfrac{H(1;q)}{q^{n+1}} dq \right)
\end{align*}
as $n \to \infty$, and therefore
\begin{align*}
\dfrac{1}{2\pi i} \int_{C\backslash\widetilde{C}} \dfrac{\sum_{j=0}^{b-1} \zeta_b^{-aj} H(\zeta_b^j; q)}{q^{n+1}} dq = o\left( \dfrac{1}{2\pi i} \int_{\widetilde{C}} \dfrac{H(1;q)}{q^{n+1}} dq \right).
\end{align*}
On $\widetilde{C}$ we have by (2) that $\sum_{j=0}^{b-1} \zeta_b^{-aj} H(\zeta_b^j; q) = H(1;q) + o\left( H(1;q) \right)$, from which it follows that
\begin{align*}
\dfrac{1}{2\pi i} \int_{\widetilde{C}} \dfrac{\sum_{j=0}^{b-1} \zeta_b^{-aj} H(\zeta_b^j; q)}{q^{n+1}} dq \sim \dfrac{1}{2\pi i} \int_{\widetilde{C}} \dfrac{H(1;q)}{q^{n+1}} dq
\end{align*}
as $n \to \infty$. Therefore, combining the estimates along $\widetilde{C}$ and $C\backslash\widetilde{C}$ we have by (3) that
\begin{align*}
c(a,b;n) \sim \dfrac{1}{b} \left[ \dfrac{1}{2\pi i} \int_{\widetilde{C}} \dfrac{H(1;q)}{q^{n+1}} dq \right] \sim \dfrac{1}{b} c(n)
\end{align*}
as $n \to \infty$. This proves the first claim. If we now assume $H(1;q)$ and $ H(\zeta_b^j;q)$ satisfy the hypotheses of Proposition \ref{WrightCircleMethod}, then it is clear that each of (1) -- (3) are satisfied and the result follows by the asymptotic for $c(n)$ in Proposition \ref{WrightCircleMethod}.
\end{proof}

Using this result, we may immediately conclude asymptotic convexity for a large class of functions.
\begin{corollary}
Let $0\leq a<b$ and $b \geq 2$. Assume that $H(1;q)$ and $H(\zeta;q)$ satisfy the conditions of Proposition \ref{WrightCircleMethod}. Then for large enough $n_1,n_2$ we have that
\begin{align*}
	c(a,b;n_1)c(a,b;n_2) > c(a,b;n_1+n_2).
\end{align*}
\end{corollary}

\begin{remark}
The proof also works for the plane partition functions $\operatorname{pp}(a,b;n)$ by Wright's asymptotic formula \eqref{WrightPlanePartitions}. Higher order Tur\'{a}n inequalities for plane partitions have recently been studied by Ono and Pujahari \cite{OnPu}.
\end{remark}

\begin{proof}
	We use the description of the asymptotics of $c(a,b;n)$ from the proof of Theorem \ref{Thm: Equidistribution} for $N=1$. Then
	\begin{align*}
		c(a,b;n_1)c(a,b;n_2) = \frac{p_0^2}{b^2} (n_1n_2)^{\frac{1}{4} (-2B-3)} e^{2\sqrt{An_1} +2\sqrt{An_2}}  \left(1+O\left(\max\left(n_1^{-\frac 12}, n_2^{-\frac 12},(n_1n_2)^{-\frac 12}\right)\right)\right)
	\end{align*}
and
\begin{align*}
	c(a,b;n_1+n_2) = \frac{p_0}{b} (n_1+n_2)^{\frac{1}{4} (-2B-3)} e^{2\sqrt{A(n_1+n_2)}} \left(1+O\left((n_1+n_2)^{-\frac 12}\right)\right).
\end{align*}
Comparing the exponential growth of the main terms immediately yields the conclusion.
\end{proof}

A very similar calculation gives the following log-concavity result.
\begin{corollary}
	Let $0\leq a<b$ and $b \geq 2$. Assume that $H(1;q)$ and $H(\zeta;q)$ satisfy the conditions of Proposition \ref{WrightCircleMethod}. For large enough $n$, we have
	\begin{align*}
		c(a,b;n)^2 \geq c(a,b;n-1)c(a,b;n+1).
	\end{align*}
\end{corollary}

We consider the case of partition statistics in slightly more detail. Let $s(\lambda)$ be a partition statistic, i.e., $s$ is a map from the set of all partitions to $\mathbb{Z}$, and let
\begin{align*}
	H_s(\zeta;q) = \sum_{\lambda} \zeta^{s(\lambda)} q^{|\lambda|}.
\end{align*}
Note that $H_s(1;q) = \sum_{\lambda} q^{|\lambda|}$ is the generating function of $p(n)$.
Then by orthogonality of roots of unity we have that (see e.g. \cite{Andrews})
\begin{align}\label{eqn: splitting}
	H_s(a,b;q) \coloneqq \frac{1}{b}\sum_{j=0}^{b-1} \zeta_b^{-aj} H_s\left(\zeta_b^j;q\right) = \sum_{\substack{\lambda \\ s(\lambda) \equiv a \pmod{b}}} q^{|\lambda|}.
\end{align}
A direct corollary of Theorem \ref{Thm: Equidistribution} is the following.
\begin{corollary}\label{Corol}
	Assume that $H_s(1;q)$ and $H_s(\zeta;q)$ satisfy the conditions of Theorem \ref{Thm: Equidistribution}, and let $s(a,b;n)$ count the number of partitions of $n$ with statistic $s$ congruent to $a \pmod{b}$. Then as $n\to \infty$ we have that $s(a,b;n) \sim \frac 1b p(n)$ as $n \to \infty$. If furthermore the conditions of Proposition \ref{WrightCircleMethod} are satisfied, we have the error term
	\begin{align*}
		s(a,b;n) = \frac{1}{b}p(n) \left(1+O\left(n^{-\frac{1}{2}}\right)\right).
	\end{align*}
\end{corollary}

\section{Proofs of Theorems \ref{Thm: ranks} to \ref{Thm: cells}}\label{Section: proofs}
In this section we prove each of the theorems from the introduction in turn. Each proof relies on the asymptotic equidistribution result concluded in Theorem \ref{Thm: Equidistribution}.

\subsection{Proof of Theorem \ref{Thm: ranks}}
	In accordance with \eqref{eqn: splitting}, we have
	\begin{align*}
		\sum_{n \geq 0} N(a,b;n) q^n =  \frac{1}{b} \sum_{n \geq 0} p(n) q^n + \frac{1}{b} \sum_{j = 1}^{b-1} \zeta_b^{-aj} R\left(\zeta_b^j ; q\right),
	\end{align*}
	where 
	\begin{align*}
		R\left( \zeta; q  \right) \coloneqq \sum_{\substack{n \geq 0 \\ m \in \Z}} N(m,n) \zeta^m q^n.
	\end{align*}
	To conclude the asymptotic equidistribution in the framework presented here, one needs only check that the conditions of Theorem \ref{Thm: Equidistribution} apply. Since the asymptotics of $(q;q)_\infty^{-1}$ follow from \eqref{equ: asymptotic P major arc} and satisfy the required properties on both the major and minor arcs, one simply needs to show that 
	\begin{align*}
		R\left( \zeta_b^j; q  \right) =   o\left( (q;q)^{-1}_\infty \right) , \qquad \left\lvert R\left( \zeta_b^j; q  \right) \right\rvert < \left\lvert (q;q)_\infty^{-1} \right\rvert,
	\end{align*}
	on the major arc and minor arcs respectively. In fact, in \cite{Males} it was shown that as $z \rightarrow 0$ with positive real part we have $R( \zeta_b^j; q )  \rightarrow 0$. Thus clearly each inequality is satisfied, the assumptions of Theorem \ref{Thm: Equidistribution} (and Corollary \ref{Corol}) apply, and we conclude the result.

\subsection{Proof of Theorem \ref{Thm: crank}}
	Let $M(m,n)$ denote the number of partitions of $n$ with crank $m$. In accordance with \eqref{eqn: splitting}, we have (see e.g. \cite[equation (3.2)]{Mahlburg})
	\begin{align*}
		\sum_{n \geq 0} M(a,b;n) q^n =  \frac{1}{b}  \sum_{n \geq 0} p(n) q^n + \frac{1}{b} \sum_{j = 1}^{b-1} \zeta_b^{-aj} C\left(\zeta_b^j ; q\right),
	\end{align*}
	where 
	\begin{align*}
		C\left( \zeta; q  \right) \coloneqq \sum_{\substack{n \geq 0 \\ m \in \Z}} M(m,n) \zeta^m q^n = \frac{(q;q)_\infty}{F_1\left(\zeta;q\right) F_1\left(\zeta^{-1};q\right)}.
	\end{align*}
	We have from Theorem \ref{Thm: asymptotics F}, as $z \to 0$ in $D_\theta$ (so on the major arc), for $q=e^{-z}$ and $\zeta$ a primitive $b$-th root of unity, that
	\begin{align*}
		F_{1}\left(\zeta;e^{-z}\right)  =\frac{1}{\sqrt{1-\zeta}} \, e^{-\frac{\zeta\Phi(\zeta,2,1)}{z}}\left( 1+O\left(|z|\right) \right).
	\end{align*}
	
Equation \eqref{equ: asymptotic P major arc} implies that on the major arc we have
	\begin{align*}
		\left(e^{-z};e^{-z}\right)_\infty^{-1} = \sqrt{\frac{z}{2\pi}} e^{\frac{\pi^2}{6z}} (1+O(|z|)),
	\end{align*}
	while Lemma \ref{Lem: asymptotic P minor arc} gives us 
	\begin{align*}
		\left|\left(e^{-z};e^{-z}\right)_\infty^{-1}\right| \le (x)^{\frac12} e^{\frac{\pi^2}{6x}-\frac{\mathcal{C}}{x}},
	\end{align*}
	for some $\mathcal{C}>0$ on the minor arc. 
	
	Moreover, one may conclude in a similar way to \cite[Proof of Theorem 1.4 (2)]{BCMO} that 
	\begin{align*}
		\left\lvert C\left( \zeta_b^j; q  \right)  \right\rvert< \left\lvert (q;q)_\infty^{-1} \right\rvert
	\end{align*}	
	on the minor arcs. For the major arcs we obtain that
	$$ 	C\left( \zeta_b^j; q  \right) =  o\left((q;q)_\infty^{-1} \right)$$ 
	holds if and only if
	$$ \left(\frac{\pi^2}{3}-\varepsilon-\phi_1 -\phi_1'\right)\frac{x}{|z|^2} >\left(\phi_2 +\phi_2'\right)\frac{y}{|z|^2},$$ 
	for $\phi_1+i\phi_2 \coloneqq \zeta_b^j\Phi(\zeta_b^j,2,1)$ and $\phi_1'+i\phi_2' \coloneqq \zeta_b^{-j}\Phi(\zeta_b^{-j},2,1)$. A straightforward calculation shows that $\phi_1 =\frac{\pi^2}{6} - \frac{\pi^2 j}{b} \left(1-\frac jb\right) = \phi_1'$ and $\phi_2=\phi_2'$. Therefore, our assumption reduces to
	$$ \left(\frac{2\pi^2 j}{b}\left(1-\frac jb\right)-\varepsilon\right)\frac{x}{|z|^2} >0,$$
	which holds, since we have $b>0$, $1\leq j\leq b-1$ and $x=\operatorname{Re}(z)>0$.
	Thus all the assumptions of Theorem \ref{Thm: Equidistribution} apply, and we conclude the result.

\subsection{Proof of Theorem \ref{Thm: residual crank}}
	 In \cite[equation (2.1)]{BriLo} it was shown that the generating function of the number of overpartitions of $n$ with residual crank $m$, denoted by $\overline{M}(m,n)$, is
	\begin{align*}
		C(\zeta;q) \coloneqq \sum_{\substack{n\geq 0\\m\in \Z}} \overline{M}(m,n)\zeta^mq^n = \frac{(q^2;q^2)_\infty}{F_1(\zeta;q)F_1(\zeta^{-1};q)}.
	\end{align*}
We thus have 
\begin{align*}
	C(a,b;q) \coloneqq \sum_{n \geq 0} \overline{M}(a,b;n)q^n = \frac{1}{b}\sum_{j=0}^{b-1} \zeta_b^{-aj} C\left(\zeta_b^j;q\right).
\end{align*}
By a similar argument as before, the asymptotic behaviour toward $z=0$ is dominated by the $j=0$ term on both the major and minor arcs. If $j=0$, we have 
\begin{align*}
	\frac{(q^2;q^2)_\infty}{(q;q)_\infty^2},
\end{align*}
which, using \eqref{equ: asymptotic P major arc} and standard arguments, is seen to satisfy the conditions of Proposition \ref{WrightCircleMethod} with $B=\frac{1}{2}$, $A=\frac{\pi^2}{4}$, and $\alpha_0 = \frac{1}{2\sqrt{\pi}}$. Then applying Theorem \ref{Thm: Equidistribution} yields the claimed result.

\subsection{Proof of Theorem \ref{Thm: plane partitions}}	
	Let $\operatorname{pp}(m,n)$ the number of plane partitions of $n$ with trace $m$. We have MacMahon's classical generating function \cite{MacMahon}
	$$\sum_{n=0}^\infty \operatorname{pp}(n) q^n = \prod\limits_{n=1}^\infty \dfrac{1}{(1 - q^n)^n}$$
	and the trace generating function \cite{Stanley}
	$$\operatorname{PP}(\zeta;q) \coloneqq  \sum_{n,m\geq 0} \operatorname{pp}(m,n) \zeta^m q^n = \prod\limits_{n=1}^\infty \dfrac{1}{(1 - \zeta q^n)^n}.$$
	Following the strategy of \cite{BCMO} we have, for $q=e^{-z}$ and $\zeta$ a $b$-th root of unity, that
	\begin{small}
		\begin{align*}
		\Log (\operatorname{PP}(\zeta;q)) = - \sum_{n=1}^\infty n \Log \left( 1 - \zeta q^n \right) = \sum_{n \geq 1} n \sum_{m \geq 1} \dfrac{\zeta^m q^{nm}}{m} = \sum_{m \geq 1} \zeta^m ~ \dfrac{q^m}{m(1 - q^m)^2} = z \sum_{m \geq 1} \zeta^m ~ \dfrac{q^m}{mz(1 - q^m)^2}.
	\end{align*}
	\end{small}
	Recall that the generating function for the Bernoulli numbers $B_n$ is given by (see e.g. \cite{BFOR})
	\begin{align*}
		\sum_{n=0}^\infty \dfrac{B_n}{n!} z^n = \dfrac{z}{e^z - 1} = \dfrac{z e^{-z}}{1 - e^{-z}}.
	\end{align*}
	If we define $B(z) \coloneqq  \frac{1}{z} \sum_{n=0}^\infty \frac{B_n}{n!} z^n = \frac{e^{-z}}{1 - e^{-z}}$, then differentiating gives us $B'(z)=-\frac{e^{-z}}{(1-e^{-z})^2}$ and yields the identity
	\begin{align*}
		\dfrac{-B^\prime(mz)}{mz} = \dfrac{e^{-mz}}{mz \left( 1 - e^{-mz} \right)^2}.
	\end{align*}
	Therefore, if we set $F(z) \coloneqq  \frac{- B^\prime(z)}{z}$, we obtain
	\begin{align*}
		\Log \left(\operatorname{PP}\left(\zeta;e^{-z}\right)\right) = z \sum_{m=1}^\infty \zeta^m F(mz).
	\end{align*}
	For $\zeta = \zeta_b^a \coloneqq  e^{\frac{2\pi i a}{b}}$ a $b$-th root of unity not equal to $1$ and by substituting $m\mapsto bm+j$ for $m\in\N_{0}$, $1\leq j\leq b$ this yields
	\begin{align}\label{eqn: PP(zeta;q)}
		\Log \left(\operatorname{PP}\left(\zeta_b^a ;e^{-z}\right)\right) = z \sum_{j = 1}^{b} \zeta_b^{aj} \sum_{m=0}^\infty F\left( \left( m + \frac jb \right) bz \right).
	\end{align}
	We turn to evaluating the inner sum. We note that $F(z)$ has the Laurent expansion
	\begin{align*}
		F(z) = \dfrac{- B^\prime(z)}{z} = \sum_{n = -3}^\infty \dfrac{-(n+2) B_{n+3}}{(n+3)!} z^n.
	\end{align*}
	By Euler-Maclaurin summation we have for $c_n \coloneqq  \frac{-(n+2) B_{n+3}}{(n+3)!}$ the identity (see e.g.\@ \cite[Lemma 2.2]{BCMO})
	\begin{align}\label{eqn: EM}
		\sum_{m=0}^\infty F\left( \left( m + \frac{j}{b} \right) bz \right) \sim \dfrac{\zeta\left( 3, \frac{j}{b} \right)}{b^3 z^3} + \dfrac{I_{F,1}^*}{bz} + \dfrac{1}{12bz} \left[ \Log(bz) + \psi\left( \frac{j}{b} \right) + \gamma \right] - \sum_{n=0}^\infty  c_n \dfrac{B_{n+1}\left( \frac{j}{b} \right)}{n+1} b^n z^n
	\end{align}
as $z \to 0$ in $D_\theta$.
	Here $\zeta(s,z)\coloneqq \sum_{n=0}^\infty \frac{1}{(n+z)^s}$ is the \textit{Hurwitz zeta function}, $\psi(x)\coloneqq \frac{\Gamma'(x)}{\Gamma(x)}$ is the \textit{digamma function}, $\gamma$ is the \textit{Euler-Mascheroni constant}, and for some $A\in\R^+$ we define
	$$ I^*_{F,A} \coloneqq \int_{0}^\infty \left(F(u)-\sum_{n=n_0}^{-2}c_n u^n -\frac{c_{-1} e^{-Au}}{u}\right) du.$$
	Here and throughout we say that
	$$f(z) \sim \sum_{n=0}^\infty a_nz^n,$$
	if $f(z) = \sum_{n=0}^N a_nz^n +O(|z|^{N+1}),$ for any $N\in\N_0$. 
	Applying \eqref{eqn: EM} to \eqref{eqn: PP(zeta;q)} and using that $\sum_{j=1}^b \zeta_b^{aj}=0$, we obtain
	\begin{align*}
		\Log \left(\operatorname{PP}\left(\zeta_b^a ; e^{-z}\right)\right) &\sim z \sum_{j = 1}^{b} \zeta_b^{aj} \left[ \dfrac{\zeta\left( 3, \frac{j}{b} \right)}{b^3 z^3} + \dfrac{I_{F,1}^*}{bz} + \dfrac{1}{12bz} \left[ \Log(bz) + \psi\left( \frac{j}{b} \right) + \gamma \right] - \sum_{n=0}^\infty  c_n \dfrac{B_{n+1}\left( \frac{j}{b} \right)}{n+1} b^n z^n \right] \\ 
		&= \dfrac{1}{b^3 z^2} \sum_{j=1}^{b} \zeta_b^{aj} \zeta\left( 3, \frac{j}{b} \right) + \dfrac{1}{12b} \sum_{j=1}^{b} \zeta_b^{aj} \psi\left( \dfrac{j}{b} \right) + O(|z|).
	\end{align*}
	We have the well-known identity (see e.g.\@ \cite[equation (2.4)]{BCMO})
	\begin{align*}
		\sum_{j=1}^b \zeta_b^{aj} \psi\left( \dfrac{j}{b} \right) =  b \Log\left( 1 - \zeta_b^a \right)
	\end{align*}
	and by elementary manipulations we furthermore obtain
	\begin{align*}
		\sum_{j=1}^{b-1} \zeta_b^{aj} \zeta\left( 3, \frac{j}{b} \right) = \sum_{j=1}^{b-1} \zeta_b^{aj} \sum_{n=0}^\infty \dfrac{b^3}{(bn+j)^3} = b^3 \operatorname{Li}_3\left( \zeta_b^a \right),
	\end{align*}
	where $\operatorname{Li}_3(z) = \sum_{k=1}^\infty \frac{z^k}{k^3}$ is the \textit{third polylogarithm function}. Therefore on the major arc, we conclude by exponentiating that, for $\zeta^a_b \not = 1$, we have
	\begin{align*}
		\operatorname{PP}\left(\zeta_b^a;e^{-z}\right) &= (1 - \zeta_b^a)^{ \frac{1}{12}} e^{\frac{\operatorname{Li}_3\left(\zeta_b^a \right)}{z^2}}  \left(1+O(|z
		|)\right)
	\end{align*}
	and otherwise by \cite{Wright3}
	\begin{align*}
		\operatorname{PP}\left(1; e^{-z}\right) &= z^{\frac{1}{12}} e^{\frac{\zeta(3)}{z^2}- \kappa} \left(1 + O(|z|) \right),
	\end{align*}
where $\kappa = \zeta'(-1) <0$.
	An analogous argument to the one of \eqref{eqn: splitting} yields that $\operatorname{PP}(a,b;q)$ and $\operatorname{PP}(\zeta;q)$ are analytic such that $$\operatorname{PP}(a,b;q)=\frac 1b \sum_{j=0}^{b-1} \zeta_b^{-aj} \operatorname{PP}\left(\zeta_b^j;q\right).$$
	Comparing exponents, we see that $\operatorname{PP}(\zeta_b^a; e^{-z}) = o( \operatorname{PP}(1; e^{-z}))$, and therefore the second hypothesis of Theorem \ref{Thm: Equidistribution} is true for $\operatorname{PP}(\zeta;q)$.
	
	We now consider $\operatorname{PP}(\zeta;q)$ on the minor arc. By definition, we have
	\begin{align*}
		\dfrac{\operatorname{PP}\left(\zeta_b^a; e^{-z}\right)}{\operatorname{PP}\left(1; e^{-z}\right)} = \prod\limits_{n=1}^\infty \left( \dfrac{1 - e^{-nz}}{1 - \zeta_b^a e^{-nz}} \right)^n.
	\end{align*}
	As $z \to 0$ with $\mathrm{Re}(z) > 0$, we see that $1 - e^{-nz} \to 0$ while $1 - \zeta_b^a e^{-nz} \not \to 0$. Thus, for all $z$ on the minor arc with $|z|$ sufficiently small, we see that $\left| \frac{\operatorname{PP}(\zeta_b^a; e^{-z})}{\operatorname{PP}(1; e^{-z})} \right| < 1$. This proves that the first hypothesis of Theorem \ref{Thm: Equidistribution} holds for $\operatorname{PP}(\zeta; q)$. The third condition of Theorem \ref{Thm: Equidistribution} follows by noting that the integral of $\operatorname{PP}(1;q)$ along the major arc gives Wright's asymptotic \eqref{WrightPlanePartitions}, and so equidistribution follows by Theorem \ref{Thm: Equidistribution}.

\subsection{Proof of Theorem \ref{Thm: Betti}}
	 For $X$ a Hilbert scheme, letting
	\begin{align*}
		G_X(T;q) \coloneqq \sum_{n \geq 0} P(X;T) q^n,
	\end{align*}
	a standard argument with orthogonality of roots of unity yields 
	\begin{align}\label{eqn: Betti splitting}
		\sum_{n \geq 0} B(a,b;X) q^n= \frac{1}{b}\sum_{r=0}^{b-1}\zeta_b^{-ar} G_X\left(\zeta_b^r;q\right) .
	\end{align}
	The main result of Boccalini's thesis \cite[equation (4.1)]{Boc} states that 
	\begin{align*}
		G_{X_1} (\zeta;q) = \sum_{n \geq 0} P\left( \text{Hilb}^{n,n+1,n+2}(0) ;\zeta\right) q^n =\frac{1+\zeta^2}{(1-\zeta^2 q)(1-\zeta^4q^2)}  F_3\left(\zeta^2;q\right)^{-1}.
	\end{align*}
	By \eqref{eqn: Betti splitting} we have that
	\begin{align*}
		H_{X_1}(a,b;q) \coloneqq \sum_{n\geq 0} B(a,b;X_1) q^n =\frac{1}{b}\left(1+(-1)^a \delta_{2\mid b}\right)G_{X_1}(1;q) + \frac{1}{b}\sum_{\substack{0\leq r< b-1 \\ r\neq \frac{b}{2}}} \zeta_b^{-ar}G_{X_1}\left(\zeta_b^r;q\right).
	\end{align*}
	
	Since 
	\begin{align*}
		G_{X_1}(1;e^{-z}) = \frac{2}{(1- e^{-z})(1-e^{-2z})} (e^{-z};e^{-z})_\infty^{-1} = (e^{-z};e^{-z})_\infty^{-1} \left(\frac{1}{z^2} + \frac{3}{2z} +\frac{11}{12} +O(z)\right),
	\end{align*}
	the asymptotic behaviour is essentially controlled by the Pochhammer symbol. It is then enough to show that on the major and minor arcs, $G_{X_1}(\zeta_b^r;q) = o(G_{X_1}(1;q))$ for $\zeta_b^r \neq 1$. This follows directly from the asymptotics of $F_3$ given in Theorem \ref{Thm: asymptotics F} in a similar fashion to \cite[Theorem 1.4 (1)]{BCMO} for the major arc, and a similar calculation to the arguments of \cite{BCMO} for the minor arc. Thus toward $z=0$ on the major arc we have
	\begin{align*}
		H_{X_1}(a,b;e^{-z}) = \frac{d(a,b)}{\sqrt{2\pi} z^{\frac{3}{2}}} e^{\frac{\pi^2}{6z}} (1+O(|z|)).
	\end{align*} 
	We are left to apply Proposition \ref{WrightCircleMethod} with $A= \frac{\pi^2}{6}, B=-\frac{3}{2}$ and $\alpha_0 = \frac{d(a,b)}{\sqrt{2\pi}}$ which yields that
	\begin{align*}
		B(a,b;X_1) = \frac{ \sqrt{3} d(a,b)}{2\pi^2} e^{\pi\sqrt{\frac{2n}{3}}} \left(1+O\left(n^{-\frac{1}{2}}\right)\right),
	\end{align*}
	from which one may also conclude asymptotic equidistribution.
	
	Similarly, it is shown in \cite[equation (4.2)]{Boc} that we have
	\begin{align*}
		G_{X_2}(\zeta;q) \coloneqq&	\sum_{n \geq 0} P\left( \text{Hilb}^{n,n+2}(0) ;\zeta\right) q^n = \frac{1+\zeta^2-\zeta^2 q}{(1-\zeta^2 q)(1-\zeta^4q^2)} F_3\left(\zeta^2;q\right)^{-1}.
	\end{align*}
	An analogous argument and application of Proposition \ref{WrightCircleMethod} to the case of $X_1$ holds. Using the generating functions \cite[equation (4.15)]{Boc} (which in turn cites \cite[Corollary 5.4]{NaYo}) and \cite[Corollary 5.4]{NaYo2}
	\begin{align*}
		& G_{X_3}(\zeta;q) \coloneqq \frac{1}{(1-\zeta^2 q)(1-\zeta^4q^2)} F_3\left(\zeta^2;q\right)^{-1}, \\
	& G_{X_4} (\zeta;q) \coloneqq F_3\left(\zeta^2;q\right)^{-1} \prod_{j=1}^{m} \frac{1}{1-\zeta^{2j} q^j} ,
\end{align*}
the cases for $X_3$ and $X_4$ follow in the same way.

\subsection{Proof of Theorem \ref{Thm: cells}}
	The results \cite[Proposition 4.2]{ES} and \cite[Proposition 2.8]{Got} show that $V_{n,k}$ has a cell decomposition and that, letting $v(m,n) \coloneqq \#\{ m\text{-dimensional cells of $V_{n,k}$} \}$, we have
	\begin{align*}
		V(\zeta;q) \coloneqq \sum_{m,n \geq 0} v(m,n) \zeta^m q^n = \prod_{n \geq 1} \frac{1}{1- \zeta^{n-1}q^n} = F_3(\zeta;q)^{-1}.
	\end{align*} 
	Then by orthogonality of roots of unity, we have
	\begin{align*}
		\sum_{n \geq 0} v(a,b;n)q^n = \frac{1}{b} \sum_{j=0}^{b-1} \zeta_b^{-aj} V\left(\zeta_b^j;q\right).
	\end{align*}
	Note that the $j=0$ term corresponds to $\frac{1}{b}(q;q)^{-1}_\infty$. Combining this with Theorem \ref{Thm: asymptotics F} (2), one can show in the same way as \cite[Theorem 1.4 (1)]{BCMO} that on both the major and minor arcs, the asymptotic behaviour of the $j=0$ term dominates as $z \to 0$. An application of Corollary \ref{Corol} immediately yields the asymptotic claimed.


\begin{bibsection}
	\begin{biblist}
\bibitem{Andrews} G. Andrews, \textit{The theory of partitions}, Cambridge University Press, no.\@ 2 (1998).	

	\bib{AndGa}{article}{
		AUTHOR = {Andrews, G.},
		author={Garvan, F.},
		TITLE = {Dyson's crank of a partition},
		JOURNAL = {Bull. Amer. Math. Soc. (N.S.)},
		FJOURNAL = {American Mathematical Society. Bulletin. New Series},
		VOLUME = {18},
		YEAR = {1988},
		NUMBER = {2},
		PAGES = {167--171},
	}
	
	 \bibitem{Arken} G. Arken, \textit{Modified Bessel functions}, Mathematical Methods for Physicists, 3rd ed., Orlando,
	 FL: Academic Press (1985), 610--616.

\bib{BO}{article}{
	author = {{Bessenrodt}, C.},
	author={{Ono}, K.},
	TITLE = {Maximal multiplicative properties of partitions}
 JOURNAL = {Ann. Comb.},
FJOURNAL = {Annals of Combinatorics},
VOLUME = {20},
YEAR = {2016},
NUMBER = {1},
PAGES = {59--64},
ISSN = {0218-0006},
}

\bibitem{Boc} D. Boccalini, \textit{Homology of the three flag Hilbert Scheme}, Ph.D. Thesis, Lausanne, EPFL (2016).

\bibitem{BFG} W. Bridges, J. Franke, and T. Garnowski, \textit{Asymptotics for the twisted eta-product and applications to sign changes in partitions}, preprint.

\bibitem{Bringmann} K. Bringmann, \textit{On the explicit construction of higher deformations of partition statistics}, Duke Math. J. {\bf 144} (2008), no.\@ 2, 195--233.
		
\bibitem{BCMO} K. Bringmann, W. Craig, J. Males, and K. Ono, \textit{Distributions on partitions arising from Hilbert schemes and hook lengths}, preprint.

\bibitem{BringDous} K. Bringmann and J. Dousse, \textit{On Dyson's crank conjecture and the uniform asymptotic behaviour of certain inverse theta functions}, Trans.  Amer. Math. Soc. {\bf 368} (2016), no.\@ 5, 3141--3155.


\bibitem{BFOR} K. Bringmann, A. Folsom, K. Ono, and L. Rolen, \textit{Harmonic Maass forms and mock modular forms: theory and applications} (Vol. 64), American Mathematical Soc. (2017).

\bibitem{BGM} K. Bringmann, F. Garvan, and K. Mahlburg, \textit{Partition statistics and quasiharmonic {M}aass forms}, Int. Math. Res. Not. IMRN {\bf 1} (2009), 63--97.


\bibitem{BJMR} K. Bringmann, C. Jennings-Shaffer, K. Mahlburg, and R. Rhoades, \textit{Peak positions of strongly unimodal sequences}, Trans. Amer. Math. Soc. {\bf 372} (2019), no.\@ 10, 7087--7109.

\bib{BriLo}{article}{
	AUTHOR = {Bringmann, K.},
	author={Lovejoy, J.},
	author={Osburn, R.},
	TITLE = {Rank and crank moments for overpartitions},
	JOURNAL = {J. Number Theory},
	FJOURNAL = {Journal of Number Theory},
	VOLUME = {129},
	YEAR = {2009},
	NUMBER = {7},
	PAGES = {1758--1772},
}

\bib{BriLo2}{article}{
	AUTHOR = {Bringmann, K.},
	author={Lovejoy, J.},
	TITLE = {Rank and congruences for overpartition pairs},
	JOURNAL = {Int. J. Number Theory},
	FJOURNAL = {International Journal of Number Theory},
	VOLUME = {4},
	YEAR = {2008},
	NUMBER = {2},
	PAGES = {303--322},
}

\bibitem{BringMahl} K. Bringmann and K. Mahlburg, \textit{Asymptotic inequalities for positive crank and rank moments}, Trans. Amer. Math. Soc. {\bf 366} (2014), no.\@ 2, 1073--1094.

\bibitem{Ciolan} A. Ciolan, \textit{Equidistribution and inequalities for partitions into powers}, preprint.

\bibitem{ComMar} D. Coman and G. Marinescu, \textit{Equidistribution results for singular metrics on line bundles}, Ann. Sci. \'{E}cole Norm. Sup. (4) {\bf 48} (2015),  no.\@ 3, 497--536.

\bibitem{CraigPun} W. Craig and A. Pun, \textit{Distribution properties for $t$-hooks in partitions}, Ann. Comb., accepted for publication.

\bibitem{DawMas} M. Dawsey and R. Masri, \textit{Effective bounds for the Andrews spt-function}, Forum Math. {\bf 31} (2019), no.\@ 3, 743--767.

\bibitem{DeSalvoPak} S. DeSalvo and I. Pak, \textit{Log-concavity of the partition function}, Ramanujan J. {\bf 38} (2015), no.\@ 1, 61--73.


	\bib{dyson}{article}{
		AUTHOR = {Dyson, F.},
		TITLE = {Some guesses in the theory of partitions},
		JOURNAL = {Eureka},
		FJOURNAL = {Eureka. The Archimedeans' Journal},
		NUMBER = {8},
		YEAR = {1944},
		PAGES = {10--15},
	}

\bibitem{ES} G. Ellingsrud and S. Str\o mme, \textit{On the homology of the {H}ilbert scheme of points in the plane}, Invent. Math. {\bf 87} (1987), no.\@ 2, 343--352.

\bib{Fulton}{book}{
	AUTHOR = {Fulton, W.},
	TITLE = {Intersection theory},
	SERIES = {Ergebnisse der Mathematik und ihrer Grenzgebiete (3) [Results
		in Mathematics and Related Areas (3)]},
	VOLUME = {2},
	PUBLISHER = {Springer-Verlag, Berlin},
	YEAR = {1984},
	PAGES = {xi+470},
}

\bibitem{Garvan} F. Garvan, \textit{New combinatorial interpretations of Ramanujan's partition congruences mod 5, 7,
and 11}, Trans. Amer. Math. Soc. {\bf 305} (1988), 47--77.

\bibitem{GGOLS} N. Gillman, X. Gonzalez, K. Ono, L. Rolen, and M. Schoenbauer, \textit{From partitions to {H}odge numbers of {H}ilbert schemes of
		surfaces}, Philos. Trans. Roy. Soc. A, {\bf 378} (2020), no.\@ 2163, 20180435, 13.



\bibitem{Got} L. G\"{o}ttsche, \textit{The {B}etti numbers of the {H}ilbert scheme of points on a smooth projective surface}, Math. Ann. {\bf 286} (1990), no.\@ 1-3, 193--207.

\bibitem{GreenTao} B. Green and T. Tao, \textit{The quantitative behaviour of polynomial orbits on nilmanifolds}, Ann. of Math. (2) {\bf 175} (2012), no.\@ 2, 465--540.


\bibitem{HaKrTs} A. Hamakiotes, A. Kriegman, and W. Tsai, \textit{Asymptotic distribution of the partition crank}, Ramanujan J. {\bf 56} (2021), no.\@ 3, 1--18.

\bibitem{HarRam} G. Hardy and S. Ramanujan, \textit{Asymptotic formulae in combinatory analysis}, Proc. London Math. Soc. Ser. 2 {\bf 17}
(1918), 75--115.

\bib{HJ}{article}{
	title = {Dyson's partition ranks and their multiplicative extensions},
	author = {Hou, E.},
	author={Jagadeesan, M.},
	journal = {Ramanujan J.},
	volume = {45},
	number = {3},
	pages = {817--839},
	year = {2018},
	publisher = {Springer}
}



\bib{JS}{article}{
	AUTHOR = {Jennings-Shaffer, C.},
	TITLE = {Rank and crank moments for partitions without repeated odd parts},
	JOURNAL = {Int. J. Number Theory},
	VOLUME = {11},
	YEAR = {2015},
	NUMBER = {3},
	PAGES = {683--703},
}

\bibitem{Kamenov-Mutafchiev} E. Kamenov and L. Mutafchiev, \textit{The limiting distribution of the trace of a random plane partition}, Acta Math. Hung. {\bf 117} (2007), no.\@ 4, 293--314.

\bibitem{Katz} N. Katz, \textit{Witt vectors and a question of Entin, Keating, and Rudnick}, Int. Math. Res. Not. IMRN (2015), no.\@ 14, 5959--5975


\bibitem{KimKimSeo} B. Kim, E. Kim, and J. Seo, \textit{Asymptotics for q-expansions involving partial theta functions}, Discrete Math. {\bf 338} (2015), no.\@ 2, 180--189.

\bibitem{MacMahon} P. MacMahon, \textit{Combinatory analysis, 2 Vols.}, (Cambridge University Press, Cambridge, 1915 and 1916; Reprinted in one volume: Chelsea, New York, 1960).  

\bibitem{Mahlburg} K. Mahlburg, \textit{Partition congruences and the Andrews-Garvan-Dyson crank}, Proc. Natl. Acad. Sci. USA {\bf 102} (2005), no.\@ 43, 15373--15376.

\bibitem{Males} J. Males, \textit{Asymptotic equidistribution and convexity for partition ranks}, Ramanujan J. {\bf54} (2021), no.\@ 2, 397--413.

\bibitem{ManRol} J. Manschot and J. Zapata Rol\'{o}n, \textit{The asymptotic profile of {$\chi_y$}-genera of {H}ilbert
	schemes of points on {K}3 surfaces}, Commun. Number Theory Phys. {\bf 9} (2015), no.\@ 2, 413--436.

\bibitem{Mao} R. Mao, \textit{Asymptotic formulas for spt-crank of partitions}, J. Math. Anal. Appl. {\bf 460} (2018), no.\@ 1, 121--139.

\bibitem{Mutafchiev} L. Mutafchiev, \textit{Asymptotic analysis of expectations of plane partition statistics}, Abh. Math. Semin. Univ. Hambg. {\bf 88} (2018), 255--272.

\bib{NaYo}{book}{
	AUTHOR = {Nakajima, H.},
	author={Yoshioka, K.},
	TITLE = {Perverse coherent sheaves on blow-up. {I}. {A} quiver
		description},
	BOOKTITLE = {Exploring new structures and natural constructions in
		mathematical physics},
	SERIES = {Adv. Stud. Pure Math.},
	VOLUME = {61},
	PAGES = {349--386},
	PUBLISHER = {Math. Soc. Japan, Tokyo},
	YEAR = {2011},
}

\bibitem{NaYo2} H. Nakajima and K. Yoshioka, \textit{Perverse coherent sheaves on blow-up. {II}. {W}all-crossing
	and {B}etti numbers formula}, J. Algebraic Geom. {\bf 20} (2011), no.\@ 1, 47--100.


\bibitem{NgoRhoades} H. Ngo and R. Rhoades, \textit{Integer Partitions, Probabilities and Quantum Modular Forms}, Res. Math. Sci. \textbf{4} (2017), Paper No.\@ 17, 36 pp.

\bibitem{OnPu} K. Ono and S. Pujahari, \textit{Turan Inequalities for the plane partition function}, preprint.

\bib{OppShu}{article}{
	title={Squarefree polynomials with prescribed coefficients},
	author={Oppenheim, A.} 
	author={Shusterman, M.},
	journal={J. Number Theory},
	volume={187},
	pages={189--197},
	year={2018},
	publisher={Elsevier}
}

\bibitem{Ramanujan} S. Ramanujan, \textit{Congruence properties of partitions}, Math. Z. {\bf 9} (1921), 147--153.


\bibitem{Stanley} R. Stanley, \textit{The conjugate trace and trace of a plane partition}, J. Combinatorial Theory Ser. A {\bf 14} (1973), 53--65.

\bibitem{Wright3} E. Wright, \textit{Asymptotic partition formulae I: Plane partitions}, Q. J. Math. (1931), no. 1, 177--189.

\bibitem{Wright1} E. Wright, \textit{Stacks}, Quart. J. Math. Oxford Ser. {\bf 19} (1968), no.\@ 2, 313--320. 

\bibitem{Wright2} E. Wright, \textit{Stacks. II}, Quart. J. Math. Oxford Ser. {\bf 22} (1971), no.\@ 2, 107--116.

\bibitem{Xi} P. Xi, \textit{When Kloosterman sums meet Hecke eigenvalues}, Invent. Math. {\bf 220} (2020), no.\@ 1, 61--127.

\bibitem{Rolon} J. Zapata Rol\'{o}n, \textit{Asymptotics of crank generating functions and {R}amanujan
	congruences}, Ramanujan J. {\bf 38} (2015), no.\@ 1, 147--178.

\bib{Zhou}{article}{
	AUTHOR = {Zhou, N. },
	TITLE = {Note on partitions into polynomials with number of parts in an
		arithmetic progression},
	JOURNAL = {Int. J. Number Theory},
	FJOURNAL = {International Journal of Number Theory},
	VOLUME = {17},
	YEAR = {2021},
	NUMBER = {9},
	PAGES = {1951--1963},
}


	\end{biblist}
\end{bibsection}
\end{document}